%% file: main.tex
\documentclass{amsart}

\usepackage{lmodern}
\usepackage{url}
\usepackage{amsmath}
\usepackage{amssymb,amsmath}

\usepackage{graphicx}
\usepackage[pdftex]{color, colortbl} 

\newcommand{\rr}{\mathbb{R}}
\newcommand{\BDM}{\textrm{BDM}}
\renewcommand{\P}{\mathbb{P}}
\newcommand{\mesh}{\mathcal{C}_h}
\newcommand{\facets}{\mathcal{F}_h}
\newcommand{\jump}[1]{ {[\![ #1 ] \!]}}
\renewcommand{\div}{\operatorname{div}}
\renewcommand{\hat}[1]{\widehat{#1}}
\newcommand{\hot}{\operatorname{hot}}
\newcommand{\eff}{\operatorname{eff}}
\newcommand{\id}{\operatorname{id}}
\newcommand{\ds}{\operatorname{ds}}
\newcommand{\dx}{\operatorname{dx}}
\newcommand{\meannorm}[1]{(| u |_{r+1} + |u|_{r'+1})}

\newcommand{\Isigstar}{I^*_h}
\newcommand{\Isig}{I_h}

\usepackage[T1]{fontenc}
\usepackage{todonotes}

\usepackage{pgfplots,pgfplotstable}
\usepgfplotslibrary{groupplots}

\usepackage{booktabs}
\usepackage{siunitx}
\newcommand{\numeoc}[1]{\num[round-precision=1,round-mode=places, scientific-notation=false]{#1}}
\newcommand{\numnr}[1]{\num[round-precision=2,round-mode=places, scientific-notation=false]{#1}}
\newcommand{\numval}[1]{\num[]{#1}}
\definecolor{mygray}{gray}{0.95}

\usepackage[backend=bibtex8,doi=false,isbn=false, url=false, eprint=true, firstinits=true, style=numeric,maxbibnames=99,defernumbers=true, ]{biblatex}
\AtEveryBibitem{%
  \clearfield{note}%
}

\newtheorem{theorem}{Theorem}
\newtheorem{lemma}{Lemma}

\newtheorem{corollary}{Corollary}
\newtheorem{remark}{Remark}

\begin{document}

\title[A posteriori error estimates for mixed Laplace
eigenvalue problems]{A note on asymptotically exact a posteriori error estimates for
mixed Laplace eigenvalue problems
}

\author[P.~L.~Lederer]{Philip L. Lederer}
\address{Department of Mathematics and Systems Analysis, Aalto University,
Otakaari 1, Espoo, Finland}
\email{philip.lederer@aalto.fi}


\maketitle

\begin{abstract}
We derive optimal and asymptotically exact a posteriori error
estimates for the approximation of the Laplace eigenvalue problem. To
do so, we combine two results from the literature. First, we use the
hypercircle techniques developed for mixed eigenvalue approximations
with Raviart-Thomas Finite elements. In addition, we use the
post-processings introduced for the eigenvalue and eigenfunction based
on mixed approximations with the Brezzi-Douglas-Marini Finite element.
To combine these approaches, we define a novel additional local
post-processing for the fluxes that appropriately modifies the
divergence. Consequently, the new flux can be used to derive upper
bounds and still shows good approximation properties. Numerical
examples validate the theory and motivate the use of an adaptive mesh
refinement.

\end{abstract}

\section{Introduction}
\label{intro}

In many examples from physics to industrial applications, the solution
of eigenvalue problems plays an essential role. Similar as for
standard source problems, the Finite element method seems to be a very
promising method to discretize these problems due to its flexibility
and good approximation properties. Numerous works deal with the
analysis in general frameworks where issues such as stability,
convergence properties and a priori error estimates are discussed, see
\cite{BABUSKA1991641,MR2652780}. 

Since in general one can not assume high regularity of the
eigenfunctions on arbitrary domains \cite{MR1469972}, the requirement
for an adaptive mesh refinement strategy is obvious. Central to this
approach is the derivation of an efficient and reliable a posteriori
error estimator, as already developed for Finite element methods in
general \cite{MR1217437,MR3059294}, and for eigenvalue
problems in particular in \cite{MR1998821}.

In this work we consider the Laplace eigenvalue problem and
approximate it using a mixed method, see \cite{MR606505,
MR1722056,MR3097958}. By that we get access to the hypercircle theory,
see \cite{MR25902, MR701093}, eventually leading to asymptotically
exact upper bounds and local efficiency. However, unlike for standard
source problems, see \cite{MR2945185, MR2373174,MR2684353, MR3335498,
HSV}, a more profound approach is needed since the orthogonality of
the corresponding errors is no longer exactly satisfied. 

For eigenvalue problems this was first introduced in the work
\cite{MR4080229}, by means of the Raviart-Thomas Finite element. 
To discuss details, note that we have
\begin{align} \label{eq::psintro}
  \| \sigma_h - \sigma \|_0^2 + \| \nabla (u - u_h^{**}) \|_0^2 &= \| \sigma_h - \nabla u_h^{**}\|_0^2 - 2(\sigma_h - \sigma, \nabla(u - u_h^{**}) ),
\end{align}
where $\lambda, u, \sigma$ are the eigenvalue, eigenfunction and its
gradient, $\lambda_h, u_h, \sigma_h$ are the corresponding
approximations and $u_h^{**}$ denotes some $H^1$-conforming
post-processed function of $u_h$. The first term on the right-hand
side of \eqref{eq::psintro} is computable and can therefore be used to
define an a posteriori estimator $\eta$. The astonishing observation
in \cite{MR4080229} was then that in the case of an approximation
using the Raviart-Thomas Finite element, the second term $ 2(\sigma_h
- \sigma, \nabla(u - u_h^{**}) )$ converges of higher order.
Consequently, $\eta$ is an asymptotically exact upper bound for the
errors on the left-hand side of \eqref{eq::psintro}.

Unfortunately, the method of \cite{MR4080229} has the drawback of a
reduced accuracy of the eigenvalue and the eigenfunction since the
Raviart-Thomas space does not allow an optimal approximation. 
In \cite{bdmest} (using ideas from \cite{MR3864690}) the same authors
(and collaborators) were able to achieve an optimal approximation by
using the Brezzi-Douglas-Marini (BDM) Finite element instead. However,
this was only possible by paying the price of unknown constants in the
a posteriori estimates since the additional term in
\eqref{eq::psintro} is not of higher order any more as was also
observed in \cite{Bertrand2019APE}.

The goal of this work is to combine the advantages from both works.
More precisely, we use the BDM Finite Element and the post-processing
techniques for the eigenvalue and the eigenfunction as in
\cite{bdmest}, and consider modifications of the approaches from
\cite{MR4080229} to derive asymptotically exact upper bounds. For the
latter, we introduce an additional (local) post-processing for the
flux variable $\sigma_h$, where we correct its divergence to fit the
additional term in~\eqref{eq::psintro}, which consequently
converges again with higher order.

The rest of the paper is organized as follows.
Section~\ref{sec::problem} discusses the problem setting and its
approximation. In section~\ref{sec::localpost} we present the local
post-processing technique for the eigenfunction and the eigenvalue.
The main results are then discussed in section~\ref{sec::aposteriori}.
While we first recapture the standard a posteriori error analysis
based on \eqref{eq::psintro} and reveal its breakdown due to
a slow convergence of the additional terms, we then introduce the
novel post-processing of the flux and derive the asymptotically exact
upper bound. In the last section~\ref{sec::numex} we present two
numerical examples to validate our findings. The appendix, see
section~\ref{sec::appendix}, considers some additional results needed
in the analysis.

\section{Problem setting} \label{sec::problem}

Let $\Omega \subset \rr^d$ be a polygon or polyhedron for $d = 2, 3$,
respectively. We consider the mixed formulation of the Laplace
eigenvalue problem with homogeneous Dirichlet boundary conditions,
i.e. we want to find a $\lambda \in \rr, u \in L^2(\Omega)$ and
$\sigma \in H(\div, \Omega)$ such that 
\begin{subequations} \label{eq::mixedeq}
  \begin{alignat}{2} 
  -(\sigma, \tau) - (\div\tau, u) &= 0 &&\quad \forall \tau \in H(\div, \Omega),\label{eq::mixedeq_one}\\
  - (\div\sigma, v) &= \lambda (u,v)  &&\quad \forall v \in L^2(\Omega).\label{eq::mixedeq_two}
  \end{alignat}    
\end{subequations}
We approximate \eqref{eq::mixedeq} by a mixed method using the $\BDM$
Finite element for the approximation of $\sigma$ and a piece-wise
polynomial approximation of $u$. To this end let $\mesh$ be a regular
triangulation of $\Omega$ into triangles and tetrahedrons in two and
three dimensions respectively. Let $k \ge 1$ be a fixed integer (see
Remark~\ref{rem::lowestorder} for a comment regarding the lowest order
case). We introduce the spaces
\begin{align*}
  U_h &:= \{v_h \in L^2(\Omega): v_h|_K \in \P^k(K) ~\forall K \in \mesh \},\\
  \Sigma_h &:=  \{\tau_h \in H(\div, \Omega): {\tau_h}|_K \in \P^{k+1}(K, \rr^d) ~\forall K \in \mesh \},
\end{align*}
where $\P^l(K)$ denotes the space of polynomials of order $l \ge 0$ on
$K$, and $\P^l(K, \rr^d)$ denotes the corresponding vector-valued
version. An approximation of \eqref{eq::mixedeq} then seeks $\lambda_h
\in \rr$, $u_h \in U_h$ and $\sigma_h \in \Sigma_h$ such that
\begin{subequations} \label{eq::mixedeqapprox}
  \begin{alignat}{2} 
  -(\sigma_h, \tau_h) - (\div\tau_h, u_h) &= 0 &&\quad \forall \tau_h \in \Sigma_h,\label{eq::mixedeqapprox_one}\\
  - (\div\sigma_h, v_h) &= \lambda_h (u_h,v_h)  &&\quad \forall v_h \in U_h. \label{eq::mixedeqapprox_two}
  \end{alignat}    
\end{subequations}
Review article \cite{MR2652780} (for example) states that problem
\eqref{eq::mixedeqapprox} defines a well approximation of the
continuous eigenvalue problem \eqref{eq::mixedeq} in the sense that it
does not produce any spurious modes and that eigenfunctions are
approximated with the proper multiplicity. The approximation results
are summarized in the following. To this end let $s > 1/2$ and let
$(\lambda, u, \sigma)$ be a solution of the eigenvalue problem
\eqref{eq::mixedeq} with the regularity $u \in H^{1 + s}(\Omega)$ and
$\sigma \in H(\div, \Omega) \cap H^s(\Omega, \rr^d)$ (for the
regularity results see \cite{MR2559736,MR775683}). Then there exists a
discrete solution of \eqref{eq::mixedeqapprox}  such that
\begin{subequations} \label{eq::approx_primal}
\begin{align} 
  \| u - u_h \|_0 &\lesssim h^{r} | u |_{r+1}, \label{eq::approx_primal_two}\\
  \| \sigma - \sigma_h \|_0 &\lesssim h^{r'} | u |_{r'+1},\label{eq::approx_primal_three} \\
  \| \div (\sigma - \sigma_h) \|_0 &\lesssim h^{r} \meannorm{u}, \label{eq::approx_primal_two_div}
\end{align}
\end{subequations}
where $r = \min\{s, k+1\}$ and $r' = \min\{s, k+2\}$.
If
$s$ is big enough we have $r' = r+1$. Above estimates follow from the
abstract theory from \cite{MR2652780}, \cite{MR1722056} and
\cite{MR606505}, and the approximation results of the source problem,
see \cite{MR3097958}. In addition we have 
\begin{align} \label{eq::approx_primal_grad}
  \| u - u_h \|_{1,h} \lesssim h^{r-1} | u |_{r+1}, 
\end{align}
where 
\begin{align*}
  \| u - u_h \|_{1,h}^2 := \sum_{K \in \mesh} \| \nabla(u
- u_h) \|_{0,K}^2 +  \sum_{F \in \facets} \frac{1}{h_F} \|
\jump{u_h} \|_{0,F}^2.
\end{align*} 
Here $\jump{\cdot}$ denotes the standard jump
operator, $\facets$ the set of facets of the triangulation $\mesh$,
and $h_F$ the diameter of a facet $F \in \facets$.
 Note that above results demand a careful choice of the approximated
eigenfunction $u_h$ and the approximated gradient $\sigma_h$. An
example, well established in the literature, is given by a
normalization such that $\| u_h \|_0 = \| u \|_0 = 1$ and by choosing
the sign $(u,u_h) > 0$. Note that this also fixes $\sigma$ and
$\sigma_h$ by \eqref{eq::mixedeq_one} and
\eqref{eq::mixedeqapprox_one}, respectively. The case of eigenvalues
with a higher multiplicity demands more carefulness, particularly if
an a posteriori analysis is considered, and we refer to
\cite{MR2652780, MR3647956} for more details. For simplicity, we
assume for the rest of this work that $\lambda$ is a simple eigenvalue
and that the above choice of sign and scaling of the continuous and
the discrete eigenfunctions is applied. Further, for simplicity, we
will call $(\lambda, u, \sigma)$ \textbf{the solution} of
\eqref{eq::mixedeq}, keeping in mind that a different scaling and sign
can be chosen. 

\begin{remark} \label{rem::lowestorder} Although the schemes proposed
  in this work are computable also for the lowest order case $k=0$,
  one does not observe any high-order convergence of the post
  processed variables defined later in the work. The reason for this
  is that the Aubin-Nitsche technique, needed in the analysis, can not
  be applied for this case.
\end{remark}

\section{Local post-processing for $u_h$ and $\lambda_h$} \label{sec::localpost}

For a sufficiently smooth solution, estimates
\eqref{eq::approx_primal} and \eqref{eq::approx_primal_grad} show that
there is a gap of two between the order of convergence of $\| \sigma -
\sigma_h \|_0$ and $\| u - u_h \|_{1,h}$. 
In \cite{MR1722056} the folloowing identity is proven
\begin{align} \label{eq::lamminlamh}
  \lambda - \lambda_h = \| \sigma - \sigma_h\|_0^2 - \lambda_h \| u - u_h \|_0^2.
\end{align}
which, due to \eqref{eq::approx_primal}, gives the estimate (using $r
\le r'$)
\begin{align} 
  | \lambda - \lambda_h | &\lesssim h^{2r} | u |_{r+1} + h^{2r'} | u |_{'r+1} \lesssim h^{2r} \meannorm{u} \label{eq::approx_primal_one}.
\end{align}
We see that the order of convergence of $| \lambda - \lambda_h |$ is
dominated by the order of the $L^2$-error of the eigenfunction.
The reduced convergence of $u_h$ (compared to the $L^2$-error of
$\sigma$) is well known for mixed methods and can be improved by means
of a local post-processing, see \cite{AB, PPRAIRO}, and
particularly for eigenfunctions in \cite{MR2669393}. Consequently,
using the ideas from \cite{MR3864690}, we can then also get an
improved eigenvalue. 

For a given integer $l \ge 0$ let $\Pi^l$ denote the $L^2$-projection
onto element-wise polynomials of order $l$.
Consider the spaces
\begin{align*}
  U_h^* := \{v_h \in L^2(\Omega): v_h|_K \in \P^{k+2}(K) ~\forall K \in \mesh \}, \quad \textrm{and} \quad 
  U_h^{**} := U_h^* \cap H^1_0(\Omega),
\end{align*} 
then we define $u^*_h \in U_h^*$ by
\begin{subequations}\label{eq::post}
  \begin{alignat}{2} 
    (\nabla u_h^*, \nabla v_h^*)_K &= (\sigma_h, \nabla v^*_h)_K &&\quad \forall v_h^* \in (\id - \Pi^k)|_K \P^{k+2}(K), \forall K \in \mesh,\label{eq::post_one}\\
    \Pi^k u_h^* & = u_h.\label{eq::post_two}
  \end{alignat}
\end{subequations}
For the discretization of the standard source problem (i.e. the
Poisson problem), it is known that the kernel inclusion property $\div
\Sigma_h  \subseteq U_h$ (see \cite{MR3097958}) and commuting
interpolation operators yield a super convergence property of the
projected error $\| \Pi^k u - u_h \|_0$ given by
$\rho(h)\mathcal{O}(h^{r'})$. Here $\rho(h)$ is a function that
depends on the regularity of the problem and for which we have
$\rho(h) \rightarrow 0$ as $h \rightarrow 0$. For convex domains we
have $\rho(h) = \mathcal{O}(h)$.
This super convergence of
the projected error is the fundamental ingredient to derive the
enhanced approximation properties of $u_h^*$.

Unfortunately the same technique, i.e. the one from the standard
source problem, does not work for the eigenvalue problem and an
improved convergence estimate of $\| \Pi^k u - u_h \|_0$ is more
involved. This has been discussed for the lowest order case in
\cite{MR2559736}, for a more general setting
including eigenvalue clusters in \cite{MR3647956}, for Maxwell
eigenvalue problems in \cite{MR3918688} and for the Stokes problem for
example in \cite{MR3864690}. Unfortunately, these results are only
presented using the full $\| \cdot \|_{\div}$-norm (or the
corresponding mixed norm) for $\Sigma$ and $\Sigma_h$. While such an
estimate is applicable for an approximation of
\eqref{eq::mixedeqapprox} using Raviart-Thomas Finite elements, the
$\BDM$ case is not covered since \eqref{eq::approx_primal_two_div} and
\eqref{eq::approx_primal_three} show different convergence orders
which would spoil the estimate.  As the author is not aware of a
detailed proof that can be found in the literature, it will be given
in the appendix in section~\ref{sec::appendix}. Note however, that
these results are already used for example in \cite{bdmest} (without
proof). The resulting super convergence reads as 
\begin{align} \label{eq::super}
  \| \Pi^k u - u_h \|_0 \lesssim \rho(h)( h \| u - u_h \|_0 + \| \sigma - \sigma_h\|_0),
\end{align} 
which in combination with the techniques from \cite{PPRAIRO}, then
yield the approximation properties (see also \cite{ESpaper})
\begin{subequations} \label{eq::approxpost}
\begin{align} 
  \| u - u_h^* \|_0 &\lesssim \rho(h) h^{r'} \meannorm{u}, \\
   \| u - u_h^*\|_{1,h} &\lesssim h^{r'} \meannorm{u}.
\end{align}
\end{subequations}
Since $u_h^*$ is not $H^1$-conforming the final post-processing step
consists of the application of an averaging operator $I^a: U_h^*
\rightarrow U_h^{**}$ often also called Oswald operator, see
\cite{MR1248895} and \cite{DiPietroErn} for details on the
approximation properties. We set $u_h^{**} := I^a(u_h^*)$ for which we
have by \eqref{eq::approxpost}
\begin{subequations} \label{eq::approxpostpost}
  \begin{align}
    \| u - u_h^{**} \|_0 &\lesssim \rho(h) h^{r'} \meannorm{u}, \\ 
 \| \nabla (u - u_h^{**}) \|_0 &\lesssim h^{r'} \meannorm{u}.
  \end{align}
  \end{subequations}
We conclude this section by introducing a post-processing of the
eigenvalue. As in \cite{MR3864690,bdmest} we define 
\begin{align} \label{eq::lambdapost}
  \lambda_h^* := \frac{(\div \sigma_h, u_h^*)}{(u_h^*,u_h^*)}.
\end{align}
The following lemma was given in \cite{MR3864690}. Since we need some
intermediate steps in the sequel, we include the proof. 
\begin{lemma} \label{lem::lampostapriori} Let $s > 1/2$ and let
  $(\lambda, u, \sigma)$ be the solution of \eqref{eq::mixedeq} with
  the regularity $u \in H^{1 + s}(\Omega)$ and $\sigma \in H(\div,
  \Omega) \cap H^s(\Omega, \rr^d)$. Further let $\| u_h^* \|_0 \neq
  0$. There holds
  \begin{align*}
    | \lambda - \lambda_h^* | \lesssim (\rho(h) h^{r'+r} + h^{2r'}) \meannorm{u},
  \end{align*}
  where $r = \min\{s,k+1\}$ and $r' = \min\{s,k+2\}$.
\end{lemma}
\begin{proof}
   Since $\|u\|_0 = 1$ we have using that
  $\div \Sigma_h \subseteq U_h$ and \eqref{eq::post_two}
  \begin{align*}
    (\sigma,\sigma) &= -(\div \sigma,u) = \lambda (u,u) = \lambda, \\
    (\sigma_h,\sigma_h) &= (\div \sigma_h,u_h) = (\Pi^k \div \sigma_h,u_h) = (\div\sigma_h, u^*_h) = \lambda_h^* (u_h^*,u_h^*),\\
    \| \sigma - \sigma_h \|^2_0 &= (\sigma - \sigma_h ,\sigma - \sigma_h )
    = (\sigma,\sigma) + (\sigma_h,\sigma_h) - 2(\sigma,\sigma_h) \\
    &= \lambda + \lambda_h^*  (u_h^*,u_h^*) + 2 (\div \sigma_h, u).
  \end{align*}
  Using $\lambda_h^*\| u - u_h^* \|^2_0 = \lambda_h^*(u,u) +
  \lambda_h^*(u_h^*,u_h^*) - 2 \lambda_h^*(u,u_h^*)$ we have in total
  \begin{align*}
    \lambda - &\lambda_h^* \\
    &= \| \sigma - \sigma_h \|^2_0 - \lambda_h^*  (u_h^*,u_h^*) - 2 (\div \sigma_h, u) - \lambda_h^*,\\
    &=\| \sigma - \sigma_h \|^2_0 + \lambda_h^*  (u,u) - 2\lambda_h^* (u,u_h^*) - \lambda_h^* \| u - u_h^* \|^2_0 - 2 (\div \sigma_h, u) - \lambda_h^*, 
  \end{align*}
  and thus again with $\| u \|_0 = 1$ this gives
  \begin{align} \label{eq::midstep}
    \lambda - \lambda_h^* &= \| \sigma - \sigma_h \|^2_0 - \lambda_h^* \| u - u_h^* \|^2_0 - 2 (\div \sigma_h + \lambda_h^* u_h^*, u).
  \end{align}
  Since $(\div\sigma_h + \lambda_h^* u_h^*, u_h^*) = 0$ (according
  to the definition of $\lambda_h^*$), the last term can be written as
  \begin{align*}
    (\div\sigma_h &+ \lambda_h^* u_h^*, u) \\
    &= (\div \sigma_h + \lambda_h^* u_h^*, u - u_h^*), \\ 
    &= (\div(\sigma_h - \sigma), u - u_h^*)  + (\div\sigma + \lambda_h^* u_h^*, u - u_h^*),\\ 
    &= (\div(\sigma_h - \sigma), u - u_h^*)  + (-\lambda u + \lambda_h^* u_h^*, u - u_h^*)\\ 
    &= (\div(\sigma_h - \sigma), u - u_h^*)  + \lambda_h^* (u_h^* - u, u - u_h^*) - (\lambda - \lambda_h^*) (u, u - u_h^*).
  \end{align*}
  By the Cauchy-Schwarz inequality we finally get 
  \begin{align*}
    | \lambda - \lambda_h^* | \le& \| \sigma - \sigma_h \|^2_0 + \lambda_h^* \| u - u_h^* \|^2_0 \\
    &+ 2 \| \div(\sigma_h - \sigma) \|_0 \|u - u_h^*\|_0 + |\lambda - \lambda_h^*| \|u - u_h^*\|_0 \\
    \lesssim& \| \sigma - \sigma_h \|^2_0 + \| u - u_h^* \|^2_0 + \| \div(\sigma_h - \sigma) \|_0\| u - u_h^*\|_0  + | \lambda - \lambda_h^* |^2.
  \end{align*}
Thus, for $h$ small enough, the last term can be moved to the left
hand side, and we can conclude the proof using \eqref{eq::approxpost}
and \eqref{eq::approx_primal}.
\end{proof}

\section{A posteriori analysis} \label{sec::aposteriori}

In this section we provide an a posteriori error analysis and define
an appropriate error estimator. 
We follow \cite{MR4080229} where the authors derived an error
estimator using the variables $\sigma_h$ and $u_h^{**}$. While this
works for a mixed approximation of \eqref{eq::approx_primal} using the
Raviart-Thomas Finite element of order $k$ (as was done in
\cite{MR4080229}), this does not work for our
setting. To discuss the problematic terms and to motivate our
modification, we present more details in the following. Since
$\sigma = \nabla u$ we have 
\begin{align*}
  \| \sigma_h - \nabla u_h^{**}\|_0^2 &= \| \sigma_h - \sigma + \sigma- \nabla u_h^{**}\|_0^2 \\
  &= \| \sigma_h - \sigma \|_0^2 + \| \nabla (u - u_h^{**}) \|_0^2 + 2(\sigma_h - \sigma, \nabla(u - u_h^{**}) ).
\end{align*}
Using integration by parts, $u_h^{**} \in H^1_0(\Omega)$ and
$-\div \sigma_h = \lambda_h u_h$, the last term can be written as 
\begin{align*}
 (\sigma_h - \sigma, \nabla(u - u_h^{**}) ) &= -(\div(\sigma_h - \sigma), u - u_h^{**} )  \\
 &= -(\lambda_h u_h - \lambda u, u - u_h^{**} ) \\
 &= -(\lambda_h u_h + \lambda u_h - \lambda u_h - \lambda u, u - u_h^{**} ) \\
 &= -(\lambda_h - \lambda) ( u_h, u - u_h^{**} )  -  \lambda ( u_h - u, u - u_h^{**}).
\end{align*}
In total this gives the guaranteed upper bound
\begin{align*}
  \| \sigma_h &- \sigma \|_0^2 + \| \nabla (u - u_h^{**}) \|_0^2 \le \\
  &\| \sigma_h - \nabla u_h^{**}\|_0^2  
  + 2| \lambda_h - \lambda| \| u - u_h^{**}\|_0 + 2 \lambda \|u_h - u\|_0 \|u - u_h^{**}\|_0.
\end{align*}
In \cite{MR4080229} the first term on the right hand side is the
(computable) proposed error estimator, whereas the second and third
are high-order terms. Compared to our setting we can see the problem
since 
\begin{align*}
  \| \sigma_h - \nabla u_h^{**}\|_0^2 & \lesssim h^{2k+4}, \\
  | \lambda_h - \lambda|  \| u - u_h^{**}\|_0 &\lesssim h^{3k+4}, \\
  \|u_h - u\|_0 \|u - u_h^{**}\|_0 & \lesssim h^{2k+4},
\end{align*} 
where for simplicity, i.e. to allow a simpler comparison, we assumed a
smooth solution. Whereas the second term converges with an increased
rate (compared to $2k+4$), the bad convergence order of $\| u -
u_h\|_0$, see equation \eqref{eq::approx_primal_one}, spoils the
estimate of the last term. As we can see in the proof above, the
problem can be traced back to
the identity $-\div \sigma_h = \lambda_h u_h$, since this is the point in
the proof where $u_h$ appears first. 

To fix this problem we propose another post-processing. Whereas the
first two post-processing routines where used to increase the
convergence rate of the error of the eigenfunction and eigenvalue i.e.
$u_h^{*}$ (and $u_h^{**}$) and $\lambda_h^*$, respectively, we now aim
to construct a $\sigma_h^*$ with a fixed divergence constraint rather
than improving its approximation properties measured in the
$L^2$-norm.  To this end we define
the space
\begin{align*}
  \Sigma_h^* := \{ \tau_h \in H(\div, \Omega): \tau_h|_K \in \P^{k+3}&(K, \rr^d) ~\forall K \in \mesh, \\
   &\tau_h\cdot n|_F \in \P^{k+1}(F)~\forall F \in \facets \}.
\end{align*}
The space $\Sigma_h^*$ reads as a $\BDM$ space of order $k+3$ with a
reduced polynomial order of the normal traces. Note that other
choices of $\Sigma_h^*$ are possible, see Remark~\ref{rem::Sigmaplus}.
The basic idea now is to find a $\sigma_h^* \in \Sigma_h^*$ being as
"close" as possible to $\sigma_h$ (i.e. being a good approximation)
such that the divergence is modified appropriately using the
additional high-order normal-bubbles (i.e. functions with a zero
normal component). Since these bubbles are defined locally, this can
be done in an element-wise procedure.  
Now let $\xi_h \in \Sigma_h^*$ be arbitrary. Proposition 2.3.1 in
\cite{MR3097958} shows that the following degrees of freedom (here
applied to $\xi_h$)
\begin{subequations} \label{eq::bdmmoments}
  \begin{alignat}{2}
    \textrm{facet moments:}& ~~  \int_F \xi_h \cdot n r_h \ds && \quad \forall r_h \in \P^{k+1}(F) ~ \forall F \in \facets, \label{eq::bdmmoments_edge}\\
    \textrm{div moments:}& ~~ \int_K \div\xi_h q_h \dx & &\quad \forall q_h \in \P^{k+2}(K) / \P^{0}(K) ~\forall K \in \mesh,\label{eq::bdmmoments_div}\\
    \textrm{vol moments:}& ~~ \int_K \xi_h  \cdot l_h \dx && \quad \forall l_h \in \mathbb{H}^{k+3}(K) ~\forall K \in \mesh, \label{eq::bdmmoments_vol}
  \end{alignat}
\end{subequations}
where $\mathbb{H}^{k+3}(K) := \{l_h \in  \P^{k+3}(K, \rr^d): \div l_h
= 0, l_h\cdot n|_{\partial K} = 0\}$, are unisolvent.  
By that we can define the post processed flux $\sigma^*_h \in
\Sigma_h^*$ by 
\begin{subequations} \label{eq::bdmmoments_def}
   \begin{alignat}{2}
    & \int_F (\sigma^*_h - \sigma_h) \cdot n r_h \ds && \quad \forall r_h \in \P^{k+1}(F)~ \forall F \in \facets, \label{eq::bdmmoments_edge_def}\\
    & \int_K (\div\sigma^*_h + \lambda_h u_h^*) q_h \dx & &\quad \forall q_h \in \P^{k+2}(K) / \P^{0}(K) ~\forall K \in \mesh,\label{eq::bdmmoments_div_def}\\
    & \int_K (\sigma^*_h - \sigma_h)  \cdot l_h \dx && \quad \forall l_h \in \mathbb{H}^{k+3}(K) ~ \forall K \in \mesh. \label{eq::bdmmoments_vol_def}
  \end{alignat}
    \end{subequations}  
Note that since $\sigma_h$ is normal continuous, i.e. the normal trace
coincides on a common facet of two neighboring elements, the boundary
constraints \eqref{eq::bdmmoments_edge_def} of $\sigma_h^*$ can be set
locally on each element (boundary) separately. Further, since
$\sigma_h \cdot n$ and $\sigma_h^* \cdot n $ have the same polynomial
degree $k+1$, the moments from \eqref{eq::bdmmoments_edge_def} result
in $\sigma_h \cdot n = \sigma^*_h \cdot n$. In
Remark~\ref{rem::sigmapostdiv} we also make a comment on the choice of
\eqref{eq::bdmmoments_div_def}. 

\begin{theorem} \label{th::sigmapostapriori}
  Let $\sigma_h^* \in \Sigma_h^*$ be the function defined by
  \eqref{eq::bdmmoments_def}, then there holds
  \begin{align*}
    -\div \sigma_h^* = \lambda_h u_h^*.
  \end{align*}
    Let $s > 1/2$ be the solution of the eigenvalue problem
  \eqref{eq::mixedeq} with the regularity $\sigma \in H(\div, \Omega)
  \cap H^s(\Omega, \rr^d)$. There holds the a priori error estimate
  \begin{align*}
    \| \sigma - \sigma_h^* \|_0 \lesssim h^{r'} \meannorm{u},
  \end{align*}
  where  $r' = \min\{s,k+2\}$ and  $r = \min\{s,k+1\}$.
\end{theorem}
\begin{proof}
  We start with the proof of the divergence identity. Let $K \in
  \mesh$ and $q_h \in \P^{k+2}(K)$ be arbitrary, then we have
\begin{align*}
 -\int_K \div \sigma_h^* q_h \dx &= -\int_K \div \sigma_h^* (q_h - \Pi^0q_h) \dx  - \int_K \div \sigma_h^*  \Pi^0p_h \dx =\\
  & = \int_K \lambda_h u_h^* (q_h - \Pi^0q_h) \dx - \Pi^0q_h \int_{\partial K} \sigma_h^* \cdot n \ds,
\end{align*}
where the second step followed due to \eqref{eq::bdmmoments_div_def} and
the Gauss theorem. Using \eqref{eq::bdmmoments_edge_def} and
\eqref{eq::mixedeqapprox_two}, the last integral can be written as 
\begin{subequations}\label{eq::lambdachoice}
\begin{align}
  -\Pi^0q_h  \int_{\partial K} \sigma_h^* \cdot n \ds &= -\Pi^0q_h \int_{\partial K} \sigma_h \cdot n \ds 
  = -\int_K \Pi^0q_h \div \sigma_h  \dx \\ 
  &= \int_K \Pi^0q_h \lambda_h u_h \dx =\int_K \Pi^0q_h \lambda_h u^*_h \dx, 
\end{align}
\end{subequations}
where we used \eqref{eq::post_two} in the last step. All together this
gives 
\begin{align*}
  -\int_K \div \sigma_h^* q_h \dx = \int_K \lambda_h u^*_h q_h \dx,
\end{align*}
from which we conclude the proof as $\div \sigma_h^* - \lambda_h
u^*_h \in \P^{k+2}(K)$ and $q_h$ was arbitrary. 


Now let $\Isigstar$ be the canonical interpolation operator into
$\Sigma_h^*$ with respect to the moments \eqref{eq::bdmmoments}, and
let $\Isig$ be the interpolation operator into $\Sigma_h$ which is
defined using the same moments \eqref{eq::bdmmoments} but with $q_h
\in \P^{k}(K) / \P^{0}(K)$ and $l_h \in \mathbb{H}^{k+1}(K)$ instead.
First, the triangle inequality gives $ \| \sigma - \sigma_h^* \|_0 \le
\| \sigma - \Isigstar \sigma\|_0  + \| \Isigstar \sigma - \sigma_h^*
\|_0$. Since the first term can be bounded by the properties of
$\Isigstar$, we continue with the latter which can be written as 
\begin{align*}
  \| \Isigstar \sigma - \sigma_h^*\|_0 = \|\Isigstar (\sigma - \sigma_h^*)\|_0 \le \| (\Isigstar - \Isig) (\sigma - \sigma_h^*) \|_0 +  \|\Isig(\sigma - \sigma_h^*)\|_0.
\end{align*}
By the definition of the interpolation operators and similar steps as
above we have $\Isig(\sigma_h^*) = \sigma_h$, and thus the term most
to the right simplifies to 
\begin{align*}
  \|\Isig(\sigma - \sigma_h^*)\|_0 = \|\Isig\sigma - \sigma_h\|_0 \le \|\Isig\sigma - \sigma\|_0 + \|\sigma - \sigma_h\|_0.
\end{align*}
We continue with the other term. For this let $\psi^{\div}_i$ be the
hierarchical dual basis functions of the highest order divergence
moments from \eqref{eq::bdmmoments_div} given by $\int_K \div(\cdot)
q_i \dx$ with $q_i \in \P^{k+2}(K) / \P^{k}(K)$. Similarly let
$\psi_i^{\mathbb{H}}$ be the hierarchical dual basis functions of the
highest order vol moments from \eqref{eq::bdmmoments_vol} given by
$\int_K (\cdot) \cdot l_i \dx$ with $l_i \in \mathbb{H}^{k+3}(K) /
\mathbb{H}^{k+1}(K)$. An explicit construction of these basis
functions can be found for example in \cite{Beuchler2012,zaglmayr2006high}. Also let
$N_{\div}$ and $N_\mathbb{H}$ be the corresponding index sets. Using
\eqref{eq::mixedeq_two}, \eqref{eq::bdmmoments_div_def} and
\eqref{eq::bdmmoments_vol_def}, this then gives
\begin{align*}
  (\Isigstar - \Isig) (\sigma &- \sigma_h^*) \\
  &= \sum_{i \in N_{\div}} \int_K \div(\sigma - \sigma_h^*) q_i \dx \psi^{\div}_i + \sum_{i \in N_{\mathbb{H}}} \int_K (\sigma - \sigma_h^*) l_i \dx \psi^\mathbb{H}_i\\
   &= - \sum_{i \in N_{\div}} \int_K (\lambda u - \lambda_h u_h^*) q_i \dx \psi^{\div}_i + \sum_{i \in N_{\mathbb{H}}} \int_K (\sigma - \sigma_h) l_i \dx \psi^\mathbb{H}_i,
\end{align*}
which implies that (using that the norms of the $q_i, l_i$ and
$\psi^{\div}_i,\psi^{\mathbb{H}}_i$ is bounded)
\begin{align*}
  \|(\Isigstar - \Isig) (\sigma - \sigma_h^*)\|_0 &\lesssim \| \lambda u - \lambda_h u_h^* \|_0 + \| \sigma - \sigma_h \|_0 \\
 & \lesssim |\lambda | \| u - u_h^*\|_0 + | \lambda - \lambda_h | \| u_h^*\|_0 + \| \sigma - \sigma_h \|_0.
\end{align*}
Since $\| u_h^*\|_0 \le \| u_h^* - u\|_0 + \| u \|_0$, we can conclude
the proof by the approximation properties of $\Isig$ and $\Isigstar$  
(see Proposition 2.5.1 in \cite{MR3097958}), estimates
\eqref{eq::approx_primal_one} and \eqref{eq::approxpost} and by $\rho(h)
h^{r'} \le h^{r'}$ and $h^{2r} \le h^{r'}$. 
\end{proof}


\begin{remark} \label{rem::Sigmaplus} Instead of choosing $\Sigma_h^*$
  as above, one can for example also use the standard Raviart-Thomas
  space of order $k+2$ denoted by $RT^{k+2}$. Since $\div RT^{k+2} =
  U_h^*$ it is again possible to set $-\div \sigma_h^* = \lambda_h
  u_h^*$ (using the appropriate degrees of freedom). However, since
  the normal trace of $\sigma^*_h$ is now in $\P^{k+2}(F)$ on each facet
  $F \in \facets$, one has to be more careful defining the edge
  moments. Precisely, we would now set 
  \begin{align*}
    \Pi^{k+1} (\sigma^*_h \cdot n) = \sigma_h \cdot n, \quad \textrm{and} \quad (\id - \Pi^{k+1})( \sigma^*_h \cdot n) = 0.
  \end{align*}  
  where the projection has to be understood as the $L^2$-projection on
  the facets. 
\end{remark}

\begin{remark}\label{rem::sigmapostdiv} One might be curious why we do
  not use $\lambda_h^*$ instead of $\lambda_h$ in the definition of
  $\div \sigma^*_h$ in \eqref{eq::bdmmoments_div_def}. Indeed, as can be
  seen in the proof this is a crucial choice since we used in
  \eqref{eq::lambdachoice} that the mean value of the divergence is
  fixed by the constant normal moments (first equal sign) and thus
  coincides with $\Pi^0 (\lambda_h u_h)$ (third equal sign). Choosing
  $\lambda_h^*$ in \eqref{eq::bdmmoments_edge_def} would then lead to a
  mismatch of the low-order and high-order parts of the divergence.
\end{remark}

We are now in the position of defining the local error estimator on
each element $K \in \mesh$ by
\begin{align*}
  \eta(K) := \| \nabla u_h^{**} - \sigma_h^* ||_K,
\end{align*}
and the corresponding global estimator by 
\begin{align*}
  \eta := \Big(\sum_{K \in \mesh} \eta(K)^2\Big)^{1/2} = \| \nabla u_h^{**} - \sigma_h^* ||_0.
\end{align*}
\begin{theorem} \label{th::aposteriori} Let $(\lambda,u,\sigma)$ be
the solution of \eqref{eq::mixedeq}. Let $(\lambda_h,u_h,\sigma_h)$
the the solution of \eqref{eq::approx_primal} and let $u_h^{**}$ and
$\sigma_h^*$ be the post-processed solutions. There holds the
reliability estimate 
\begin{align*}
  \| \nabla u - \nabla u_h^{**} \|_0^2 + \| \sigma - \sigma^*_h\|^2_0 \le \eta^2 + \hot(h) 
\end{align*}
where $\hot(h) := 2 |(\sigma_h^* - \sigma, \nabla(u-u_h^{**}))|$ 
with 
\begin{align*}
  \hot(h) \lesssim \rho(h) (h^{2r + r'} + \rho(h)h^{2r'}) \meannorm{u},
\end{align*} is a
 high-order term compared to $\mathcal{O}(h^{2r'})$ as $h
 \rightarrow 0$. Further, there holds the efficiency 
 \begin{align*}
   \eta \le \| \nabla u - \nabla u_h^{**} \|_0 + \| \sigma - \sigma^*_h\|_0.
 \end{align*}
\end{theorem}
\begin{proof}
  Following the same steps as at the beginning of this section we arrive at 
  \begin{align*}
    \| \nabla u - \nabla u_h^{**} \|_0^2 + \| \sigma - \sigma^*_h\|^2_0 = \| \nabla u_h^{**} - \sigma_h^* ||^2_0 + 2 (\sigma_h^* - \sigma, \nabla(u-u_h^{**})).
  \end{align*}
For the last term we now have
\begin{align*}
  (\sigma^*_h - \sigma, \nabla(u - u_h^{**}) ) &= -(\div(\sigma^*_h - \sigma), u - u_h^{**} )  \\
  &= -(\lambda_h u^*_h - \lambda u, u - u_h^{**} ) \\
  &= -( \lambda_h - \lambda) ( u^*_h, u - u_h^{**} )  -  \lambda ( u^*_h - u, u - u_h^{**}).
 \end{align*}
 Whereas the first term converges of order
 \begin{align*}
  | \lambda_h - \lambda| |( u^*_h, u - u_h^{**} )| &\le | \lambda_h - \lambda| \| u_h^*\|_0 \|u - u_h^{**}\|_0 \\
  &\lesssim \rho(h)h^{2r+r'} \meannorm{u},
 \end{align*}
 we have for the second term
 \begin{align*}
  |\lambda| |( u^*_h - u, u - u_h^{**})| &\le |\lambda| \| u_h^* - u \|_0 \| u - u_h^{**} \|_0\\
  &\lesssim \rho(h)^2 h^{2r'}\meannorm{u}.
 \end{align*}
 It remains to show that $\hot(h) \lesssim \rho(h) (h^{2r + r'} +
 \rho(h)h^{2r'})$ is of higher order compared to $h^{2r'}$. Due to the
 additional $\rho(h)$ in the upper bound of $\hot(h)$, we only have to
 show that $2r'  \le 2r + r'$. For the low regularity case, i.e. $s =
 r = r'$, and the case of full regularity, i.e. $r = k+1$ and $r' = k
 + 2$, this follows immediately. For the case where $r = k+1$ and $r'
 = s$ with $k+1 < s < k+2$, we also have 
 \begin{align*}
   2 r' = 2 s < k + 2 + s < 2(k+1) + s = 2r + r',
 \end{align*}
 from which we conclude the proof of the reliability.
 
 The efficiency estimate follows by the triangle inequality and $\sigma = \nabla u$. 
\end{proof}


Using the estimator from above we are now also able to derive an upper
bound for $\lambda_h^*$. To this end let 
\begin{align*}
  \eta_\lambda := \eta^2 + \| \sigma_h - \sigma_h^* \|^2_0 + | ( \lambda_h^* u_h^* - \lambda_h u_h, u_h^{**}) |.
\end{align*}
The last two terms from the estimator $\eta_\lambda$ are needed to
measure the difference between the quantities used in $\eta$ and the
functions used in the definition of $\lambda_h^*$. Unfortunately the
authors do not see how the definition of $\lambda_h^*$ can be changed
such that only $\sigma_h^*$ and $u_h^{**}$ are used, which would allow
a direct estimate by $\eta$.
\begin{theorem} \label{lem::lambdaapost} Let $(\lambda,u,\sigma)$ be the
  solution of \eqref{eq::mixedeq}. Let $(\lambda_h,u_h,\sigma_h)$ the
  the solution of \eqref{eq::approx_primal} and let $u_h^{**}$ and
  $\sigma_h^*$ be the post-processed solutions. There holds the
  estimate
  \begin{align*}
    | \lambda - \lambda_h^* | \lesssim  \eta_\lambda + \hot(h) + \widetilde{\hot}(h),
  \end{align*}
  where $\widetilde{\hot}(h) := \|u_h^* - u\|_0 \|u - u_h^{**}\|_0 +
  \|u - u_h^{**}\|^2_0$ 
  with 
  \begin{align*}
    \widetilde{\hot}(h)  \lesssim \rho(h)^2h^{2r'} \meannorm{u},
  \end{align*}  
  and $\hot(h)$ are higher order terms compared to $\mathcal{O}(\rho(h)h^{r + r'}
  + h^{2r'})$ as $h \rightarrow 0$.
\end{theorem}
\begin{proof}
    According to \eqref{eq::midstep} we have the
    equation 
  \begin{align} \label{eq::tobound}
    \lambda - \lambda_h^* &= \| \sigma - \sigma_h \|^2_0 - \lambda_h^* \| u - u_h^* \|^2_0 - 2 (\div \sigma_h  + \lambda_h^* u_h^*, u).
  \end{align}
  Note that the second term on the right side is already of higher
  order, thus we only consider the rest. The idea is to modify the
  terms including $\sigma_h$ such that we can use the results from the
  previous theorem. By the triangle inequality we have $\| \sigma -
  \sigma_h \|_0 \le \| \sigma - \sigma^*_h \|_0 + \| \sigma_h^* -
  \sigma_h \|_0$. Since the error $\| \sigma_h^* - \sigma_h \|_0$ is
  computable and $\| \sigma - \sigma^*_h \|_0$ can be bounded by the
  estimator from the previous theorem, we are left with an estimate
  for the last term on the right hand side of \eqref{eq::tobound}.

  In contrast to the the proof of Lemma~\ref{lem::lampostapriori} we
  now add and subtract $u_h^{**}$  (and not $u_h^*$) which gives 
  \begin{align*}
    (\div\sigma_h + \lambda_h^* u_h^*, u) 
    &= (\div\sigma_h + \lambda_h^* u_h^*, u-u_h^{**}) + (\div\sigma_h + \lambda_h^* u_h^*, u_h^{**}) \\
    &= (\div\sigma_h + \lambda_h^* u_h^*, u-u_h^{**}) + ( \lambda_h^* u_h^* - \lambda_h u_h, u_h^{**}). 
  \end{align*}
  The last term is computable and will be used in the estimator. For
  the first one we have using that $u_h^{**} \in H_0^1(\Omega)$ and
  integration by parts 
  \begin{align*}
    (\div \sigma_h + \lambda_h^* u_h^*, &u - u_h^{**})\\
    =& (\div(\sigma_h - \sigma), u - u_h^{**})  + (\div\sigma + \lambda_h^* u_h^*, u - u_h^{**}),\\ 
    =& -(\sigma_h - \sigma, \nabla(u - u_h^{**}))  + (-\lambda u + \lambda_h^* u_h^*, u - u_h^{**}),\\ 
    \le&  \| \sigma_h - \sigma \|_0^2 + \|\nabla(u - u_h^{**})\|_0^2  \\
    &+ \lambda_h^* \|u_h^* - u\|_0 \|u - u_h^{**}\|_0 + |\lambda - \lambda_h^*| \|u - u_h^{**}\|_0, \\
    \le&  \| \sigma_h - \sigma \|_0^2 + \|\nabla(u - u_h^{**})\|_0^2  \\
    &+ \lambda_h^* \|u_h^* - u\|_0 \|u - u_h^{**}\|_0 + |\lambda - \lambda_h^*|^2 + \|u - u_h^{**}\|^2_0.
  \end{align*}
  The first term can be estimated as before, thus for $h$ small enough
  we have 
  \begin{align*}
    |\lambda - \lambda_h^* | 
    \lesssim & \| \sigma - \sigma_h^* \|^2_0 + \|\nabla(u - u_h^{**})\|_0^2 + \| \sigma_h - \sigma_h^* \|^2_0 \\
    & + | ( \lambda_h^* u_h^* - \lambda_h u_h, u_h^{**}) | +  \widetilde{\hot}(h),\\
    \lesssim &\eta^2 + \| \sigma_h - \sigma_h^* \|^2_0 + | ( \lambda_h^* u_h^* - \lambda_h u_h, u_h^{**}) |  + \hot(h) + \widetilde{\hot}(h).
  \end{align*}
  To show that $\hot(h)$ and $\widetilde{\hot}(h)$ are of higher order
  compared to $\mathcal{O}(\rho(h)h^{r + r'} + h^{2r'})$, one follows
  the same steps as in the proof of Theorem~\ref{th::aposteriori}.

\end{proof}

\section{Numerical examples}\label{sec::numex}

In this section we discuss some numerical examples to validate our
theoretical findings. All methods were implemented in the Finite
element Library Netgen/NGSolve, see \url{www.ngsolve.org} and
\cite{netgen}.

\subsection{Convergence on a unit square}
The first example considers the unit square domain $\Omega = (0,1)^2$.
The eigenfunction and the smallest eigenvalue of \eqref{eq::mixedeq}
is given by $u = 2 \sin(2\pi x) \sin(2\pi y)$ and $\lambda = 2 \pi^2$,
respectively. We start with an initial mesh with $|\mesh| = 32$
elements and use a uniform refinement. Note that for simplicity we
used a structured mesh for this example, thus we have $h \sim (0.5
|\mesh|)^{-1/2}$. In Table~\ref{tab::unitsquare_one} and
Table~\ref{tab::unitsquare_two} we present several errors and their
convergence rate (given in brackets) for different polynomial orders
$k = 1$ and $k = 2$.  Beside the errors we also plot the high-order
term from Theorem~\ref{th::aposteriori}, and the efficiencies 
\begin{align*}
  \eff := \frac{\eta^2}{\| \nabla u - \nabla u_h^{**} \|_0^2 + \| \sigma - \sigma^*_h\|^2_0}, 
  \qquad \textrm{and} \qquad \eff_{\lambda} := \frac{\eta_\lambda}{| \lambda - \lambda_h^*|}.
\end{align*}
Since $\Omega$ is convex we have for this example that $\rho(h) \sim
h$, thus we expect the following convergence orders (for simplicity
recalled here)
\begin{alignat*}{2}
  \| u - u_h^{**} \|_0 &\lesssim h^{k+3}, \qquad & \| \nabla(u - u_h^{**}) \|_0 &\lesssim h^{k+2}, \\
  \| \sigma - \sigma_h^* \|_0 &\lesssim h^{k+2}, &\qquad \| \lambda - \lambda_h^*\|_0 &\lesssim h^{2(k+2)}.
\end{alignat*}
In accordance to the theory all errors converge with the optimal
orders. Further the high-order term $\hot(h)$ converges faster than
the estimator $\eta$ as predicted by Theorem~\ref{th::aposteriori}.
Note that this results in an efficiency $\eff$ converging to one, i.e.
the error estimator is asymptotically exact. Also the estimator for
the error of the eigenvalue converges appropriately and shows a good
efficiency $\eff_\lambda$. The same conclusions can be made for $k=2$,
however, the error of the eigenvalues $\lambda_h$ and $\lambda_h^*$
converge so fast that rounding errors dominate on the finest meshes.
For the same reason we also do not present any numbers for
$\widetilde{\hot}(h)$ since this term converges even faster resulting
in very small numbers already on coarse meshes.

\setlength\tabcolsep{2pt}
  \begin{table}[]
    \centering  
    \begin{tabular}{c|cc|cc|cc|cc|c} 
      \toprule\toprule
      \multicolumn{10}{c}{$ k = 1$} \\
      $|\mesh|$ & \multicolumn{2}{c|}{$\|\nabla (u - u^{**}_h)\|_0$}&\multicolumn{2}{c|}{$\|\sigma - \sigma_h^{*}\|_0$}&\multicolumn{2}{c|}{$ \eta$}&\multicolumn{2}{c|}{$\hot(h)$}& $ \eff$ \\
      [0.8ex]\hline
      \rule{0pt}{2.6ex}32 &\numval{0.030715961769110837} & (--) &\numval{0.025021054990437597} & (--) &\numval{0.0373577616476502} & (--) &\numval{0.007890441608203545} & (--) &\numnr{0.8891885690866133}\\
      128 &\numval{0.003951367376824789}&~(\numeoc{2.9585646655824696}) &\numval{0.003043805895026751}&~(\numeoc{3.0391943552068046}) &\numval{0.0049204316191424815}&~(\numeoc{2.924551237058139}) &\numval{0.0005255354341360917}&~(\numeoc{3.9082461004729985}) &\numnr{0.9731726993368174}\\
      512 &\numval{0.0005016610656721495}&~(\numeoc{2.977567107743334}) &\numval{0.0003774283522188177}&~(\numeoc{3.0116016538021677}) &\numval{0.0006261069474856197}&~(\numeoc{2.974303858964148}) &\numval{3.339090568793944e-05}&~(\numeoc{3.976260910528664}) &\numnr{0.994656202093606}\\
      2048 &\numval{6.314681821193486e-05}&~(\numeoc{2.989931026029959}) &\numval{4.708826521665002e-05}&~(\numeoc{3.0027633232388204}) &\numval{7.874128981782394e-05}&~(\numeoc{2.9912168601212112}) &\numval{2.096302209912933e-06}&~(\numeoc{3.9935366045837415}) &\numnr{0.9992530572654734}\\
      8192 &\numval{7.917947945180342e-06}&~(\numeoc{2.9955115547275186}) &\numval{5.8838123415851745e-06}&~(\numeoc{3.0005444340504206}) &\numval{9.865163452675455e-06}&~(\numeoc{2.99670548488259}) &\numval{1.3119380841498311e-07}&~(\numeoc{3.9980751757368775}) &\numnr{1.0000853185274075}\\
      32768 &\numval{9.911683036745303e-07}&~(\numeoc{2.9979246249739426}) &\numval{7.354468204399958e-07}&~(\numeoc{3.0000583037233874}) &\numval{1.2343141475144359e-06}&~(\numeoc{2.9986333305584436}) &\numval{8.205841633041038e-09}&~(\numeoc{3.9989045130171643}) &\numnr{1.000154128122538}\\
      \rowcolor{mygray}
      \multicolumn{10}{c}{$ k = 2$} \\
      \rowcolor{mygray}
      $|\mesh|$ & \multicolumn{2}{c|}{$\|\nabla (u - u^{**}_h)\|_0$}&\multicolumn{2}{c|}{$\|\sigma - \sigma_h^{*}\|_0$}&\multicolumn{2}{c|}{$ \eta$}&\multicolumn{2}{c|}{$\hot(h)$}& $ \eff$ \\
      [0.8ex]\hline
      \rowcolor{mygray}
      \rule{0pt}{2.6ex}32 &\numval{0.002581914659682153} & (--) &\numval{0.0015400046637293664} & (--) &\numval{0.002920121078979839} & (--) &\numval{0.00034776553814828824} & (--) &\numnr{0.9434834762794305}\\
      \rowcolor{mygray}
      128 &\numval{0.00016047447794771287}&~(\numeoc{4.008025543105617}) &\numval{9.969587196683122e-05}&~(\numeoc{3.9492571404851375}) &\numval{0.00018702252497057297}&~(\numeoc{3.9647442460893307}) &\numval{1.0551080604375399e-05}&~(\numeoc{5.042652308418008}) &\numnr{0.9799979374749206}\\
      \rowcolor{mygray}
      512 &\numval{9.987006212391515e-06}&~(\numeoc{4.006147788802603}) &\numval{6.298832261474089e-06}&~(\numeoc{3.9843774715187474}) &\numval{1.1758069779102775e-05}&~(\numeoc{3.991488888672532}) &\numval{3.2538254805108196e-07}&~(\numeoc{5.019110078111011}) &\numnr{0.9916553371839669}\\
      \rowcolor{mygray}
      2048 &\numval{6.229950819614326e-07}&~(\numeoc{4.002759589119348}) &\numval{3.94870051806111e-07}&~(\numeoc{3.9956345341190214}) &\numval{7.361957588484842e-07}&~(\numeoc{3.9974179960640703}) &\numval{1.0119620755810393e-08}&~(\numeoc{5.006909744353087}) &\numnr{0.9962116484044109}\\
      \rowcolor{mygray}
      8192 &\numval{3.890475667232457e-08}&~(\numeoc{4.001202312833273}) &\numval{2.4700244597785798e-08}&~(\numeoc{3.998780719899514}) &\numval{4.60419228253006e-08}&~(\numeoc{3.9990694481445144}) &\numval{3.157540133027067e-10}&~(\numeoc{5.002210343098384}) &\numnr{0.9981995811249716}\\
      \rowcolor{mygray}
      32768 &\numval{2.4306210673713028e-09}&~(\numeoc{4.000549655884054}) &\numval{1.5441503566614368e-09}&~(\numeoc{3.9996401860438056}) &\numval{2.878375133267173e-09}&~(\numeoc{3.999621552063743}) &\numval{1.0132590959510407e-11}&~(\numeoc{4.961726039578611}) &\numnr{0.9991225989701048}\\
      \bottomrule\bottomrule\end{tabular}        
    \caption{Convergence of several errors for the example on the unit square with $k=1, 2$} \label{tab::unitsquare_one}
\end{table}
\begin{table}[]
  \centering  
  \begin{tabular}{c|cc|cc|cc|c} 
    \toprule\toprule
    \multicolumn{8}{c}{$ k = 1$} \\
    $|\mesh|$ & \multicolumn{2}{c|}{$\|u - u^{**}_h\|_0$}&\multicolumn{2}{c|}{$ |\lambda - \lambda^*_h|$}&\multicolumn{2}{c|}{$ \eta_{\lambda} $}& $ \eff_{\lambda}$ \\
    [1ex]\hline
    \rule{0pt}{2.6ex}32 &\numval{0.001477381078555108} & (--) &\numval{0.00045239870616242683} & (--) &\numval{0.002393131438635574} & (--) &\numnr{5.289872420139851}\\
    128 &\numval{9.911522198709917e-05}&~(\numeoc{3.897791553792137}) &\numval{7.818597364206425e-06}&~(\numeoc{5.854541178021819}) &\numval{4.11228438225566e-05}&~(\numeoc{5.86281578979026}) &\numnr{5.25961907321346}\\
    512 &\numval{6.327094844547469e-06}&~(\numeoc{3.969491514214248}) &\numval{1.2544873939646095e-07}&~(\numeoc{5.961739936406396}) &\numval{6.627973657831473e-07}&~(\numeoc{5.955228357625409}) &\numnr{5.283411925635066}\\
    2048 &\numval{3.9791673351036045e-07}&~(\numeoc{3.991004747554661}) &\numval{1.9774617499024316e-09}&~(\numeoc{5.98730437442919}) &\numval{1.0463040875356019e-08}&~(\numeoc{5.98519376038985}) &\numnr{5.2911470352699705}\\
    8192 &\numval{2.4918151754543213e-08}&~(\numeoc{3.997197600742146}) &\numval{4.497024974625674e-11}&~(\numeoc{5.458535082488823}) &\numval{1.640948923569098e-10}&~(\numeoc{5.994628059006002}) &\numnr{3.6489655557354075}\\
    32768 &\numval{1.5583350783005037e-09}&~(\numeoc{3.999119678939386}) &\numval{6.345857173073455e-11}&~(\numeoc{-0.4968441545565766}) &\numval{2.598042691334991e-12}&~(\numeoc{5.9809613862647515}) &\numnr{0.040940768449043034}\\
    \rowcolor{mygray}
    \multicolumn{8}{c}{$ k = 2$} \\
    \rowcolor{mygray}
    $|\mesh|$ & \multicolumn{2}{c|}{$\|u - u^{**}_h\|_0$}&\multicolumn{2}{c|}{$ |\lambda - \lambda^*_h|$}&\multicolumn{2}{c|}{$ \eta_{\lambda} $}& $ \eff_{\lambda}$ \\
    [1ex]\hline
    \rowcolor{mygray}
    \rule{0pt}{2.6ex}32 &\numval{7.952371554390279e-05} & (--) &\numval{4.063176124446954e-06} & (--) &\numval{1.4316239925619102e-05} & (--) &\numnr{3.523411116609598}\\
    \rowcolor{mygray}
    128 &\numval{2.4890027908139703e-06}&~(\numeoc{4.997745411675775}) &\numval{1.6264667834775537e-08}&~(\numeoc{7.964722732237317}) &\numval{5.8869669612426514e-08}&~(\numeoc{7.925912380833501}) &\numnr{3.6194818246799416}\\
    \rowcolor{mygray}
    512 &\numval{7.768669931477828e-08}&~(\numeoc{5.001756420757176}) &\numval{6.864908641546208e-11}&~(\numeoc{7.888285122836398}) &\numval{2.330582378344989e-10}&~(\numeoc{7.980690212519017}) &\numnr{3.394921185462511}\\
    \rowcolor{mygray}
    2048 &\numval{2.4273126343760045e-09}&~(\numeoc{5.0002357743551995}) &\numval{1.857358711276902e-11}&~(\numeoc{1.8859880515056524}) &\numval{1.097091014951575e-12}&~(\numeoc{7.730863480973912}) &\numnr{0.059067266236221216}\\
    \rowcolor{mygray}
    8192 &\numval{7.586214658352145e-11}&~(\numeoc{4.999835932792865}) &\numval{6.88196166720445e-11}&~(\numeoc{-1.8895673866730724}) &\numval{1.7757539716379002e-13}&~(\numeoc{2.627179599898025}) &\numnr{0.0025803020381530784}\\
    \rowcolor{mygray}
    32768 &\numval{2.381557071188913e-12}&~(\numeoc{4.993403167781228}) &\numval{2.9260505129968806e-10}&~(\numeoc{-2.0880629138830984}) &\numval{2.0699680516063582e-13}&~(\numeoc{-0.2211767889425969}) &\numnr{0.0007074273128273111}\\
    \bottomrule\bottomrule\end{tabular}      
  \caption{Convergence of several errors for the example on the unit square with $k=1, 2$.} \label{tab::unitsquare_two}
\end{table}

\subsection{Adaptive refinement on the L-shape}

For the second example we choose the L-shape domain $\Omega = (-1,1)^2
\setminus ([0,1] \times [-1,0])$ where the first eigenvalue reads as
$\lambda \approx 9.6397238440219$. Note that all digits except the
last two have been proven to be correct, see \cite{MR2178637}. In this
example the corresponding eigenfunction is singular, thus we expect a
suboptimal convergence of order $\mathcal{O}(N^{-2/3})$ on a uniform
refined mesh, where $N$ denotes the number of degrees of freedom. To
this end we solve the problem using an adaptive mesh refinement. The
refinement loop is defined as usual by 
\begin{align*}
   \mathrm{SOLVE} \rightarrow \mathrm{ESTIMATE} \rightarrow \mathrm{MARK} \rightarrow \mathrm{REFINE} \rightarrow \mathrm{SOLVE} \rightarrow \ldots
\end{align*}
and is based on the local contributions $\eta(K)$ as element-wise
refinement indicators. In the marking step we mark an element if
$\eta(K) \geq \frac{1}{4} \max\limits_{K \in \mesh} \eta(K)$. The
refinement routine then refines all marked elements plus further
elements in a closure step to guarantee a regular triangulation. In
Figure~\ref{fig::lshape} we see the error history of the post
processed eigenvalue $\lambda_h^*$, its estimator $\eta_\lambda$ and
the estimator for the eigenfunction error $\eta$ for polynomial order
$k=2,3$. We can observe an optimal convergence
$\mathcal{O}(N^{-2(k+2)})$, $\mathcal{O}(N^{-2(k+2)})$ and
$\mathcal{O}(N^{-(k+2)})$, for $|\lambda - \lambda_h^*|$,
$\eta_\lambda$ and $\eta$, respectively. Further $\eta_\lambda$ shows
a good efficiency.

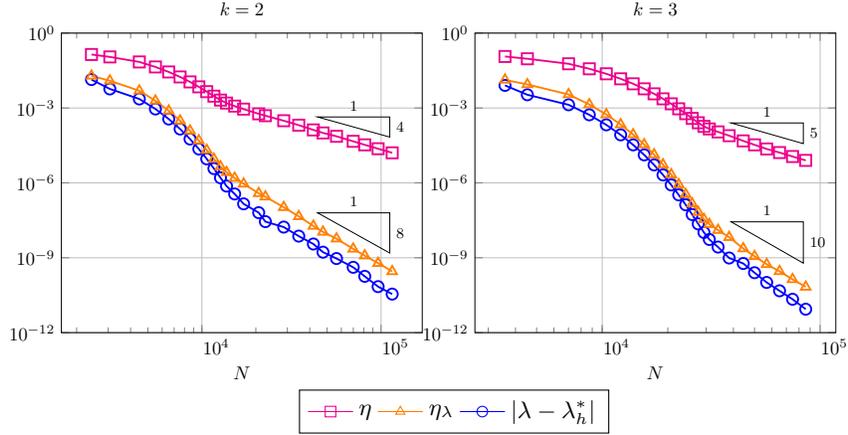
\begin{figure}[h]
  \input{lshape}
\caption{Convergence history of the L-shape example using an adaptive refinement for $k=2,3$.}\label{fig::lshape}
\end{figure}


\section{Appendix} \label{sec::appendix}

In this section we present a proof of the super convergence estimate 
\begin{align*}
  \| \Pi^k u - u_h \|_0 \lesssim \rho(h)( h \| u - u_h \|_0 + \| \sigma - \sigma_h\|_0).
\end{align*}
For this we will follow very similar steps as in \cite{MR3918688} with
several changes in order to get the proper $h$ scaling. We define the
auxillary problem: find  $\hat u_h \in U_h$ and $\hat \sigma_h \in
\Sigma_h$ such that
\begin{subequations} \label{eq::auxprob}
  \begin{alignat}{2} 
  -(\hat \sigma_h, \tau_h) - (\div \tau_h, \hat u_h) &= 0 &&\quad \forall \tau_h \in \Sigma_h,\\
  - (\div\hat \sigma_h, v_h) &= \lambda (u,v_h)  &&\quad \forall v_h \in U_h. 
  \end{alignat}    
\end{subequations}
Note that above solution provides the property 
\begin{align} \label{eq::auxprop}
  \lambda_h (\hat u_h, u_h) = - (\div \sigma_h ,\hat u_h) = (\sigma_h, \hat \sigma_h) = -(\div\hat \sigma_h, u_h) = \lambda (u,u_h). 
\end{align}
\begin{lemma} \label{lem::app_one} Let $(\lambda,u,\sigma)$ be the
  solution of \eqref{eq::mixedeq}, and let $(\hat u_h, \hat \sigma_h)$
  be the solution of \eqref{eq::auxprob}. There holds the estimate
  \begin{align*}
    \| \Pi^k u - \hat u_h \|_0 \lesssim \rho(h) ( \| \sigma - \hat \sigma_h\|_0 + h \| \div(\sigma - \hat \sigma_h) \|_0).
  \end{align*}
\end{lemma}
\begin{proof}
We solve the continuous problem: Find $\Theta \in H(\div, \Omega)$ and
$\Psi \in L^2(\Omega)$ such that
\begin{subequations} \label{eq::auxprob_two}
  \begin{alignat}{2} 
  -(\Theta, \tau) - (\div\tau, \Psi) &= 0 &&\quad \forall \tau \in H(\div, \Omega),\\
  - (\div \Theta, v) &= \lambda (\Pi^k u - \hat u_h,v)  &&\quad
  \forall v \in L^2(\Omega). 
  \end{alignat} 
\end{subequations}
  Note, that we have the regularity $\Theta \in H^s(\Omega, \rr^d)$
  and $\Psi \in H^{1+s}(\Omega)$ with $s > 1/2$, and there holds the
  stability estimate (see for example \cite{MR2559736}) 
  \begin{align} \label{eq::auxstab} 
    \| \Theta \|_s + \| \Psi \|_{1+s} \lesssim \| \Pi^k u - \hat u_h \|_0.
  \end{align}   
This then gives 
\begin{align*}
  \| \Pi^k u - \hat u_h \|_0^2 &= -(\div\Theta,\Pi^k u - \hat u_h) 
  = -(\Pi^k\div\Theta,\Pi^k u - \hat u_h), \\
  &= -(\div \Isig\Theta,\Pi^k u - \hat u_h)
  = -(\div \Isig\Theta,u - \hat u_h),
\end{align*}
where we used the commuting diagram property of the
$\BDM$-interpolation operator $\Isig$ and the $L^2$ projection
$\Pi^k$, see \cite{MR3097958}. By problems \eqref{eq::auxprob} and
\eqref{eq::auxprob_two} we  then have
\begin{align*}
  -(\div\Isig\Theta,u - \hat u_h) &= (\Isig\Theta,\sigma - \hat \sigma_h) \\
  &= (\Isig\Theta - \Theta,\sigma - \hat \sigma_h)  + (\Theta,\sigma - \hat \sigma_h) \\
  &= (\Isig\Theta - \Theta,\sigma - \hat \sigma_h)  - (\div(\sigma - \hat \sigma_h), \Psi) \\
  &= (\Isig\Theta - \Theta,\sigma - \hat \sigma_h)  + (\div(\sigma - \hat \sigma_h), \Pi^k \Psi - \Psi),
\end{align*}
where the last step followed by  $(\div(\sigma - \hat \sigma_h), \Pi^k
\Psi) = 0$. By the interpolation properties of $\Pi^k$ and
$\Isig$ and the stability \eqref{eq::auxstab} we conclude 
\begin{align*}
  \| \Pi^k u - \hat u_h \|_0^2 &\lesssim \| \sigma - \hat \sigma_h \|_0 h^s \| \Theta \|_s + \| \div(\sigma - \hat \sigma_h) \|_0 h^{1+s} \| \Psi \|_{1+s} \\
  &\lesssim h^s(\| \sigma - \hat \sigma_h \|_0  + h \| \div(\sigma - \hat \sigma_h) \|_0)\| \Pi^k u - \hat u_h \|_0.
\end{align*}
\end{proof}
\begin{lemma}
  Let $(\lambda,u,\sigma)$ be the solution of
  \eqref{eq::mixedeq}, $(\lambda_h,u_h,\sigma_h)$ be the solution of
  \eqref{eq::mixedeqapprox} and let $(\hat u_h, \hat \sigma_h)$ be the
  solution of \eqref{eq::auxprob}.
  There holds the estimate
  \begin{align*}
    \| \sigma - \hat \sigma_h \|_0  + h \| \div(\sigma - \hat \sigma_h) \|_0 \lesssim h \| u_h - \hat u_h \|_0 + h \| u - u_h \|_0 + \| \sigma - \sigma_h\|_0.
  \end{align*}
\end{lemma}
\begin{proof}
  We start with the estimate of the divergence term.
  By the triangle inequality we have 
  \begin{align*}
    \| \div(\sigma - \hat \sigma_h) \|_0 &\le \| \div(\sigma - \sigma_h) \|_0 + \| \div(\sigma_h - \hat \sigma_h) \|_0,
\end{align*}
Using $\div \Sigma_h = U_h$  gives 
\begin{align*}
  \| \div(\sigma_h - \hat \sigma_h) \|_0 &= \sup\limits_{v_h \in U_h} \frac{(\div(\sigma_h - \hat \sigma_h), v_h)}{\|v _h \|_0} \\
  &= \sup\limits_{v_h \in U_h} \frac{(\lambda_h u_h - \lambda u, v_h)}{\|v _h \|_0} \le  \| \lambda_h u_h - \lambda u \|_0,
\end{align*}
thus since also $\| \div(\sigma - \sigma_h) \|_0 = \| \lambda u  -
\lambda_h u_h\|_0$ we have with \eqref{eq::lamminlamh} and a small
enough mesh size $h$ that 
\begin{align*}
  \| \div(\sigma_h - \hat \sigma_h) \|_0 \lesssim \| \lambda_h u_h - \lambda u \|_0 
  &\lesssim  \| u -  u_h\|_0 + |\lambda - \lambda_h|  \\
  &\lesssim \| u -  u_h\|_0 + \|\sigma - \sigma_h \|_0.
\end{align*}
For the second term we proceed similarly. The triangle inequality
gives $\| \sigma - \hat \sigma_h \|_0 \le \| \sigma -  \sigma_h \|_0
+ \| \sigma_h - \hat \sigma_h \|_0$. For the latter we then have
  \begin{align*}
    \| \sigma_h - \hat \sigma_h \|_0 &= \sup\limits_{\tau_h \in \Sigma_h} \frac{( \sigma_h - \hat \sigma_h, \tau_h)}{\| \tau_h \|_0} 
    = \sup\limits_{\tau_h \in \Sigma_h} \frac{( u_h - \hat u_h, \div\tau_h)}{\| \tau_h \|_0} \\
    &= \sup\limits_{\tau_h \in \Sigma_h} \frac{\| u_h - \hat u_h \|_0 \| \div \tau_h \|_0 }{\| \tau_h \|_0} \lesssim h \| u_h - \hat u_h \|_0
  \end{align*}
  where we used that $\| \div \tau_h \|_0 \lesssim h \| \tau_h \|_0$
  which follows from standard scaling arguments.
\end{proof}
\begin{lemma}
  Let $(\lambda,u,\sigma)$ be the solution of \eqref{eq::mixedeq},
  $(\lambda_h,u_h,\sigma_h)$ be the solution of
  \eqref{eq::mixedeqapprox} and let $(\hat u_h, \hat \sigma_h)$ be the
  solution of \eqref{eq::auxprob}. There holds the estimate 
  \begin{align*}
    \| u_h - \hat u_h \|_0 \lesssim \| \Pi^k u - \hat u_h \|_0.
  \end{align*}
\end{lemma}
\begin{proof}
  Using equation \eqref{eq::auxprop} the proof follows with exactly
  the same steps as in the proof of Lemma 11 in \cite{MR3918688} or Lemma
  6.3 in \cite{MR3647956}. 
\end{proof}
Combining above results we have the super convergence property.
\begin{corollary}
  Let $(\lambda,u,\sigma)$ be the solution of \eqref{eq::mixedeq},
  $(\lambda_h,u_h,\sigma_h)$ be the solution of
  \eqref{eq::mixedeqapprox} and let $(\hat u_h, \hat \sigma_h)$ be the
  solution of \eqref{eq::auxprob}. For $h$ small enough there holds
  the super convergence property
  \begin{align*}
    \| \Pi^k u - u_h \|_0 \lesssim \rho(h)( h \| u - u_h \|_0 + \| \sigma - \sigma_h\|_0).
  \end{align*}
\end{corollary}

\printbibliography


\section*{Data availability}

All datasets generated during the current study are available in the
repository \url{https://doi.org/10.5281/zenodo.6417423}.

\section*{Funding} 
This work was supported by the Academy of Finland (Decision 324611).

\end{document}

%% file: lshape.tex
\pgfplotstableread{lshape_order_1_rawdata.dat} \dataone
\pgfplotstableread{lshape_order_2_rawdata.dat} \datatwo
\pgfplotstableread{lshape_order_3_rawdata.dat} \datathree

\newcommand{\logLogSlopeTriangle}[5]
{

    \pgfplotsextra
    {
        \pgfkeysgetvalue{/pgfplots/xmin}{\xmin}
        \pgfkeysgetvalue{/pgfplots/xmax}{\xmax}
        \pgfkeysgetvalue{/pgfplots/ymin}{\ymin}
        \pgfkeysgetvalue{/pgfplots/ymax}{\ymax}

        \pgfmathsetmacro{\xArel}{#1}
        \pgfmathsetmacro{\yArel}{#3}
        \pgfmathsetmacro{\xBrel}{#1-#2}
        \pgfmathsetmacro{\yBrel}{\yArel}
        \pgfmathsetmacro{\xCrel}{\xArel}

        \pgfmathsetmacro{\lnxB}{\xmin*(1-(#1-#2))+\xmax*(#1-#2)} 
        \pgfmathsetmacro{\lnxA}{\xmin*(1-#1)+\xmax*#1} 
        \pgfmathsetmacro{\lnyA}{\ymin*(1-#3)+\ymax*#3} 
        \pgfmathsetmacro{\lnyC}{\lnyA-#4*(\lnxA-\lnxB)}
        \pgfmathsetmacro{\yCrel}{\lnyC-\ymin)/(\ymax-\ymin)} 

        \coordinate (A) at (rel axis cs:\xArel,\yArel);
        \coordinate (B) at (rel axis cs:\xBrel,\yBrel);
        \coordinate (C) at (rel axis cs:\xCrel,\yCrel);

        \pgfmathsetmacro{\convorder}{#4*2}
        \draw[#5]   (A)-- node[pos=0.5,anchor=south] {\footnotesize  $1$}
                    (B)-- 
                    (C)-- node[pos=0.5,anchor=west] {\footnotesize  $\pgfmathprintnumber[precision=1]{\convorder}$} 
                    cycle;
    }
}

\begin{tikzpicture}
  [
  scale=0.7
  ]
  \begin{groupplot}[
    group style ={group size = 2 by 1},
    name=plot1,
    title={Convergence of the lshap=example},
    scale=1.0,
    xlabel=$N$,
    legend columns=3,
    legend style={font=\small},
    xmode=log,
    ymax=1e-0,
    ymin=1e-12,
    ymode=log,
    xticklabel style={text width=3em,align=right},
    yticklabel style={text width=3em,align=right},
    %
    grid=major,
    ]
        
    \nextgroupplot[title = {$k = 2$}, legend to name={mylegend}]    
    
    \addlegendentry{$\eta$} 
    \addlegendimage{mark=square, magenta}
    \addlegendentry{$\eta_{\lambda}$} 
    \addlegendimage{mark=triangle, orange}
    \addlegendentry{$| \lambda - \lambda_h^*|$} 
    \addlegendimage{mark=o, blue}

    \logLogSlopeTriangle{0.91}{0.2}{0.72}{2}{black};
    \logLogSlopeTriangle{0.91}{0.2}{0.4}{4}{black};


    \addplot[line width=1pt, mark=o, mark size=3, color=blue] table[x=0,y=2]{\datatwo};
    \addplot[line width=1pt, mark=square, mark size=3, color=magenta] table[x=0,y=3]{\datatwo};
    \addplot[line width=1pt, mark=triangle, mark size=3, color=orange] table[x=0,y=4]{\datatwo};

    \nextgroupplot[title = {$k = 3$}, legend to name=dummy]  



    
    \logLogSlopeTriangle{0.91}{0.2}{0.7}{2.5}{black};
    \logLogSlopeTriangle{0.91}{0.2}{0.37}{5}{black};

    \addplot[line width=1pt, mark=o, mark size=3, color=blue] table[x=0,y=2]{\datathree};
    \addplot[line width=1pt, mark=square, mark size=3, color=magenta] table[x=0,y=3]{\datathree};
    \addplot[line width=1pt, mark=triangle, mark size=3, color=orange] table[x=0,y=4]{\datathree};

\end{groupplot}

\node[right=1em,inner sep=0pt] at (4,-1.5) {\pgfplotslegendfromname{mylegend}};
\end{tikzpicture}